\documentclass[10pt]{amsart}
\usepackage{txfonts}
\theoremstyle{plain}
\newtheorem{Thm}{Theorem}

\newtheorem{Cor}[Thm]{Corollary}

\newtheorem{Lem}[Thm]{Lemma}

\begin{document}
\title[Symmetry Results for Monge-Ampere Systems on a Bounded Domain]
{Symmetry Results for classical solutions of Monge-Ampere systems on
a bounded planar domain}

\author{Li Ma, Baiyu Liu }

\address{LM: Department of mathematical sciences \\
Tsinghua university \\
Beijing 100084 \\
China} \email{lma@math.tsinghua.edu.cn}

\address{BL: Department of mathematical sciences \\
Tsinghua university \\
Beijing 100084 \\
China} \email{liuby05@mails.tsinghua.edu.cn}

\maketitle

\begin{abstract}
In this paper, by the method of moving planes, we establish the
monotonicity and symmetry properties of convex solutions for
Monge-Ampere systems on bounded smooth planar domains.

\emph{Keyword: Method of Moving plane, Monge-Ampere system}

{\em Mathematics Subject Classification: 35J60, 53C21, 58J05}
\end{abstract}


\section{Introduction}
Motivated by two famous uniqueness theorems (namely, the Cohn-Vossen
theorem and the Minkowski theorem for convex surfaces in 3-space)
and the interesting work of C.Li \cite{Cli1} on symmetric results
for Monge-Ampere equation, we investigate the monotonicity and
symmetry results of convex solutions of the Monge-Ampere system in a
bounded smooth planar domain. Similar problem in the whole plane is
considered in \cite{MaL}. In the case of a single equation,
questions of symmetry kind have been intensively studied in the
literature, see for example B. Gidas, W. M. Ni and L. Nirenberg
\cite{Gid} \cite{gnn} , W. X. Chen and C. Li\cite{ChenL} \cite{CL},
C. Li \cite{Cli1}, Y. Li and W.M. Ni \cite{LiN}, H. Berestycki and
L. Nirenberg \cite{BN}. Interesting Dirichlet problem for single
Monge-Ampere equation has been discussed by N.S.Trudinger and
J.Urbas in \cite{TU}.

The system we consider is the following:
\begin{eqnarray}
  \label{eq:sys}
  \left\{
\begin{array}{l@{\quad \quad}l}
det(D^2u)+g(u,v,\nabla u)=0, & in \ \Omega,\\
det(D^2v)+f(u,v,\nabla v)=0, & in \ \Omega,\\
(D^2u)>0,\ (D^{2}v)>0, & in \ \Omega,
\end{array}\right.
\end{eqnarray}
where $\Omega\subset \mathbb{R}^2$ is a bounded smooth domain and is convex in the $x_1$ direction.
Roughly speaking, we prove that if the equations and the boundary conditions are monotone (symmetric) in a direction then the solutions are also monotone (symmetric) in that direction.

Without loss of generality, we assume $0\in \Omega$ and
$$if\ (x_1, x_2)\in \Omega\ (x_1<0),\quad then\  (x_1', x_2)\in \Omega\ for\ x_1'\in (x_1, -x_1).$$
Here is our main result.
\begin{Thm}
\label{thm:main}
Let $(u, v)$ be a classical solution of (\ref{eq:sys}) with boundary value condition:
\begin{eqnarray}
\label{con:bd}
if & & (x_1, x_2)\in \partial \Omega,\ (x_1', x_2)\in \Omega\ and\ x_1<x_1',\\
\nonumber & & then\qquad u(x_1,x_2)>u(x_1',x_2), \quad v(x_1,x_2)>v(x_1',x_2);\\
\nonumber if & &  (x_1, x_2),\ (x_1', x_2)\in \partial \Omega\ and\ x_1<x_1',\\
\nonumber & & then \qquad u(x_1,x_2)\geq u(x_1',x_2),\quad  v(x_1,x_2)\geq v(x_1',x_2).
\end{eqnarray}
Suppose $f$ and $g$ satisfy:
$f, g\in C^1(\mathbb{R}^4)$;
\begin{eqnarray}\label{con:symfg}&   &
\begin{array}{l@{\quad \quad}l}
g(u,v,p_1,p_2)\geq g(u,v,-p_1,p_2), & \quad \forall u,v,p_2 \in \mathbb{R}, \forall  p_1<0,\\
f(u,v,p_1,p_2)\geq f(u,v,-p_1,p_2), & \quad \forall u,v,p_2 \in \mathbb{R}, \forall  p_1<0;
\end{array}\\
\label{con:monfg}&   &
\frac{\partial g}{\partial v}(u, v, p_1, p_2)>0,\quad \frac{\partial f}{\partial u}(u, v, p_1, p_2)>0 \quad \forall u,v,p_1,p_2\in \mathbb{R}.
\end{eqnarray}
Then for all $(x_1, x_2)\in \Omega$, $x_1<0$, we have:
\begin{equation*}
u(x_1, x_2)\geq u(-x_1, x_2),\ u_1(x_1, x_2)<0;\ \
v(x_1, x_2)\geq v(-x_1, x_2),\ v_1(x_1, x_2)<0.
\end{equation*}
\end{Thm}

Type example for the boundary condition in Theorem \ref{thm:main} is
when $u<0$ and $v<0$ in $\Omega$ with the Dirichlet boundary
conditions $u=0=v$ on $\partial \Omega$. Theorem \ref{thm:main} will
be proved by the use of the maximum principle and the method of
moving planes in the similar spirit as in \cite{Bu00}.

The above monotonicity result yields the following symmetry result.
\begin{Cor}\label{thm:col}
In addition to the assumptions of Theorem \ref{thm:main}, we assume
that $\Omega$ is symmetric in the $x_1$ direction, $g$ and $f$ are
symmetric in $x_1$ and $p_1$, i.e.
\begin{eqnarray*}
\begin{array}{l@{\quad \quad}l}
g(u,v,p_1,p_2)= g(u,v,-p_1,p_2), & \quad \forall u,v,p_2 \in \mathbb{R}, \ \forall  p_1<0;\\
f(u,v,p_1,p_2)= f(u,v,-p_1,p_2), & \quad \forall u,v,p_2 \in \mathbb{R}, \ \forall  p_1<0.
\end{array}
\end{eqnarray*}
Assume further that the boundary values of $u, v$ are symmetric in
the $x_1$ direction, i.e.
\begin{eqnarray*}
&if&\ (x_1, x_2)\in \partial \Omega,\ (x_1', x_2)\in \Omega,\\
& & then\ u(x_1,x_2)>u(x_1',x_2),\ v(x_1,x_2)>v(x_1',x_2);\\
&if&\ (x_1, x_2),\ (x_1', x_2)\in \partial \Omega,
\\ &&
 \ then\ u(x_1,x_2)=u(x_1',x_2),\  v(x_1,x_2)=v(x_1',x_2).
\end{eqnarray*}
Then $u, v$ are symmetric in the $x_1$ direction. That is
$$
u(-x_1, x_2)=u(x_1, x_2), \quad v(-x_1, x_2)=v(x_1, x_2), \quad \forall (x_1, x_2)\in \Omega.
$$
\end{Cor}

The plan of the paper is below. In section \ref{sect2}, we make a
Lemma for the method of moving planes. We use the method of moving
planes to prove Theorem \ref{thm:main} in section \ref{sect3}.
Corollary \ref{thm:col} is proved in section \ref{sect4}.

\section{Monge-Ampere System in bounded domains in $\mathbf{R}^{2}$}
\label{sect2}

Denote by
$$
a:=-\inf\{\ x_1\ |\ (x_1, x_2)\in \Omega\}>0.
$$
In what follows, we shall use the method of moving planes.
To proceed, we start by considering lines parallel to $x_1=0$, coming
from $-a$. For each $-a<\lambda\leq 0$, we define
$$
\Sigma(\lambda) := \{x\in\Omega\ | \ x_1 <\lambda\},\quad
T_\lambda := \{x\in \Omega \ |\ x_1=\lambda\}.
$$
For any point $x=(x_1, x_2) \in \Sigma(\lambda) $, let
$x^\lambda=(2\lambda-x^1, x_2)$ be the reflected point with respect
to the line $T_\lambda$. We define the reflected functions by
$$
u_\lambda(x):= u(x^\lambda), \quad v_\lambda(x):= v(x^\lambda),\quad x\in \Sigma(\lambda),
$$
and introduce the functions
$$
U(x, \lambda):= u_\lambda(x)-u(x), \quad V(x, \lambda):=
v_\lambda(x)-v(x), \quad x\in \Sigma(\lambda).
$$
They also can be regarded as functions of three variables $(x_1, x_2, \lambda)$ defined on
$$
Q_\mu:=\{\ (x_1, x_2, \lambda)\ |\ (x_1, x_2)\in \Sigma(\lambda), \ x_1<\lambda<\mu \}.
$$
By the definition of $u_{\lambda}$,
$$
det(D^2u_{\lambda})(x)=det(D^2u)(x^{\lambda}).
$$
By using the integral form of the mean value theorem, we obtain:
\begin{equation}\label{eq:det}
det(D^2 u_\lambda)(x)-det(D^2 u)(x)=a_{ij}(x)U_{ij}(x),
\end{equation}
where
\begin{equation}
\label{eq:aij}
(a_{ij})=\frac{1}{2}(det(D^2 u_\lambda)(D^2
u_\lambda)^{-1}+det(D^2 u)(D^2 u)^{-1}).
\end{equation}
Noticing that the solution $u, v$ are convex everywhere, we have
\begin{equation}
\label{eq:posaij} (a_{ij}(x))>0, \quad x \in \bar{\Omega}.
\end{equation}

We now compute the equations which $U(x, \lambda)$ and $V(x,\lambda)$ satisfy under an extra assumption: $\frac{\partial u_{\lambda}}{\partial x_1}(x)\leq 0$ or $\frac{\partial v_{\lambda}}{\partial x_1}(x)\leq 0$. By using (\ref{con:symfg}),
\begin{equation*}
  det(D^2u_{\lambda})(x) =-g(u, v, \nabla u)(x^{\lambda})=-g(u_{\lambda}, v_{\lambda}, -u_{\lambda,1},
  u_{\lambda,2}) \geq -g(u_{\lambda}, v_{\lambda}, \nabla u_{\lambda}).
\end{equation*}
Combining the above inequality with (\ref{eq:det}), we get
\begin{equation}\label{eq:hhh}
a_{ij}U_{ij}(x,\lambda)+g(u_{\lambda}(x), v_{\lambda}(x), \nabla u_{\lambda}(x))-g(u(x), v(x), \nabla u(x))\geq 0.
\end{equation}
Similar computations derive that: when $\frac{\partial v_{\lambda}}{\partial x_1}(x)\leq 0$,
\begin{equation}\label{eq:iii}
b_{ij}V_{ij}(x,\lambda)+f(u_{\lambda}(x), v_{\lambda}(x), \nabla v_{\lambda}(x))-f(u(x), v(x), \nabla v(x))\geq 0,
\end{equation}
where
\begin{equation}
\label{eq:bij}
(b_{ij})=\frac{1}{2}(det(D^2 v_\lambda)(D^2
v_\lambda)^{-1}+det(D^2 v)(D^2 v)^{-1}).
\end{equation}
Moreover, $b_{ij}$ is also an uniformly elliptic coefficient matrix on $\bar{\Omega}$.

By using the Taylor's expansion, we can rewrite (\ref{eq:hhh}) and (\ref{eq:iii}): at the point which satisfies $\frac{\partial u_{\lambda}}{\partial x_1}(x)\leq 0$,
\begin{eqnarray}
  \label{eq:bu}
  a_{ij}U_{ij}&+&\frac{\partial g}{\partial p_1}(u_{\lambda}, v_{\lambda}, \theta_1(x,\lambda) ,
  u_{\lambda,2})U_{1}
  +\frac{\partial g}{\partial p_2}(u_{\lambda}, v_{\lambda},
  u_{1}, \theta_{2}(x,\lambda))U_{2}\\
\nonumber  &+&\frac{\partial g}{\partial u}(\xi_1(x,\lambda), v_\lambda, \nabla u )U+\frac{\partial
  g}{\partial v}(u, \eta_{1}(x, \lambda), \nabla u)V\geq 0;
\end{eqnarray}
at the point which satisfies $\frac{\partial v_{\lambda}}{\partial x_1}(x)\leq 0$,
\begin{eqnarray}
  \label{eq:bv}
  b_{ij}V_{ij}& + &\frac{\partial f}{\partial p_1}(u_{\lambda}, v_{\lambda}, \tau_1(x,\lambda) ,
  v_{\lambda,2})V_{1}+\frac{\partial f}{\partial p_2}(u_{\lambda}, v_{\lambda},
  v_{1}, \tau_{2}(x,\lambda))V_{2}\\
\nonumber  & + & \frac{\partial f}{\partial u}(\xi_2(x,\lambda),
v_\lambda, \nabla v )U+\frac{\partial
  f}{\partial v}(u, \eta_{2}(x, \lambda), \nabla v)V\geq 0,
\end{eqnarray}
where for $i=1,2,$
$$
\xi_{i}(x, \lambda)\in (\overline{u(x),
u_{\lambda}(x)}, \quad
\eta_{i}(x, \lambda) \in  (\overline {v(x), v_{\lambda}(x)}),
$$
$$\theta_{i}(x, \lambda)\in (\overline {u_{i}(x),
u_{\lambda,i}(x)}),\quad
\tau_i(x, \lambda)\in (\overline{v_{i}(x), v_{\lambda,i}(x)}).
$$
Here we have used the notation $(\overline{A,B})$ to denote the open
interval in the line from $A$ to $B$.\\
For a positive function $\psi$ defined on $\Omega$, we introduce functions:
$$
\bar{U}(x, \lambda):=\frac{U(x, \lambda)}{\psi},\quad \bar{V}(x, \lambda):=\frac{V(x, \lambda)}{\psi}.
$$
Direct computation shows that:
at the point which satisfies $\frac{\partial u_{\lambda}}{\partial x_1}(x)\leq 0$,
\begin{eqnarray}
  \label{eq:bbu}
  a_{ij}\bar{U}_{ij}&+&(2a_{1j}\psi_j\frac{1}{\psi}+\frac{\partial g}{\partial p_1}(u_{\lambda}, v_{\lambda},\theta_1(x,\lambda) ,u_{\lambda,2}))\bar{U}_{1}\\
  \nonumber &+&(2a_{2j}\psi_j\frac{1}{\psi}+\frac{\partial g}{\partial p_2}(u_{\lambda}, v_{\lambda}, u_1 ,\theta_2(x,\lambda)))\bar{U}_{2}\\
  \nonumber &+&(a_{ij}\psi_{ij}\frac{1}{\psi}+\frac{\partial g}{\partial p_1}(u_{\lambda}, v_{\lambda},\theta_1(x,\lambda) ,u_{\lambda,2})\psi_1\frac{1}{\psi}\\
  \nonumber & \quad & +\frac{\partial g}{\partial p_2}(u_{\lambda}, v_{\lambda},u_1 ,\theta_2(x,\lambda))\psi_2\frac{1}{\psi}+\frac{\partial g}{\partial u}(\xi_1(x,\lambda), v_\lambda, \nabla u ))\bar{U}\\
  \nonumber & +& \frac{\partial g}{\partial v}(u, \eta_{1}(x, \lambda), \nabla u)\bar{V}\geq 0;
\end{eqnarray}
at the point which satisfies $\frac{\partial v_{\lambda}}{\partial x_1}(x)\leq 0$,
\begin{eqnarray}
  \label{eq:bbv}
  b_{ij}\bar{V}_{ij}& + & (2b_{1j}\psi_j\frac{1}{\psi}+\frac{\partial f}{\partial p_1}(u_{\lambda}, v_{\lambda}, \tau_1(x,\lambda) ,v_{\lambda,2}))\bar{V}_1\\
  \nonumber &+& (2b_{2j}\psi_j\frac{1}{\psi}+\frac{\partial f}{\partial p_2}(u_{\lambda}, v_{\lambda},
  v_{1}, \tau_{2}(x,\lambda))\bar{V}_2\\
  \nonumber &+& (b_{ij}\psi_{ij}\frac{1}{\psi}+\frac{\partial f}{\partial p_1}(u_{\lambda}, v_{\lambda}, \tau_1(x,\lambda) ,v_{\lambda,2})\psi_1\frac{1}{\psi}\\
  \nonumber & \quad & +\frac{\partial f}{\partial p_2}(u_{\lambda}, v_{\lambda},
  v_{1}, \tau_{2}(x,\lambda)\psi_2\frac{1}{\psi}+\frac{\partial
  f}{\partial v}(u, \eta_{2}(x, \lambda), \nabla v))\bar{V}\\
  \nonumber & + & \frac{\partial
  f}{\partial u}(\xi_{2}(x, \lambda), v_\lambda, \nabla v)\bar{U}\geq 0,
\end{eqnarray}
where $\xi_{i}$, $\eta_{i}$, $\theta_{i}$ and  $\tau_i$ are the same as in (\ref{eq:bu}) and (\ref{eq:bv}).

From (\ref{eq:aij}) and (\ref{eq:bij}), we can consider (\ref{eq:bbu}) and (\ref{eq:bbv}) as two uniformly elliptic equations with bounded coefficients. The following result for two elliptic differential operators will be used in the proof of our main theorem.
\begin{Lem}\label{lem:psi}
Let $L_1=A_{ij}(x)D_{ij}+B_i(x)D_i+C(x)$ and $L_2=B_{ij}(x)D_{ij}+E_i(x)D_i+F(x)$ be two uniformly elliptic operators on a domain $\Sigma\subset \Omega \subset \mathbb{R}^2$. That is
$$
A_{ij}\xi_i\xi_j\geq m^2|\xi|^2,\ (m>0),\quad B_{ij}\xi_i\xi_j\geq m^2|\xi|^2, \quad \forall \xi \in \mathbb{R}^2, \forall x\in \Sigma,
$$
$$
|A_{ij}|+|B_i|+|C|+|B_{ij}|+|E_i|+|F|\leq C_0.
$$
Then there exists an $\epsilon_0>0$ depends  on $m, C_0$ only, such that if $\Sigma$ lies in a narrow region in the $x_1$ direction $-a<x_1<-a+\epsilon\leq -a+\epsilon_0$, then the function $\psi(x)=\psi(x_1)=e-e^{\frac{1}{2\epsilon}(x_1+a)}$ which depends on $x_1$ satisfies the following:
\begin{eqnarray*}(i) &   & \forall (x, \lambda)\in \Sigma\times (-a, 0), \quad \psi>1,\quad and \quad \frac{L_1 \psi}{\psi}<-1,\quad
\frac{L_2 \psi}{\psi}<-1;\\
(ii) &   & \forall (y_i, \lambda_i)\in \Sigma\times (-a, -a+\epsilon_0) \ (i=1,2),\\
&   & \frac{L_1 \psi}{\psi}|_{(y_1,\lambda_1)}\frac{L_2 \psi}{\psi}|_{(y_2,\lambda_2)}
-\frac{\partial g}{\partial v}(u, \eta_{1}, \nabla u)|_{(y_1,\lambda_1)}\frac{\partial
  f}{\partial u}(\xi_{2}, v_{\lambda}, \nabla v)|_{(y_2,\lambda_2)}
>0,
\end{eqnarray*}
where $\xi_{i}$ and $\eta_{i}$ are the same as in (\ref{eq:bu}) and (\ref{eq:bv}).
\end{Lem}

\begin{proof}
It is easy to see that $e-1\geq \psi(x)\geq e-e^{\frac{1}{2}}>1$.
\begin{eqnarray*}
L_1\psi& =& -A_{11}\frac{1}{4\epsilon^2}e^{\frac{1}{2\epsilon}(x_1+a)}-B_1\frac{1}{2\epsilon}e^{\frac{1}{2\epsilon}(x_1+a)}+C(e-e^{\frac{1}{2\epsilon}(x_1+a)})\\
& \leq & -(m^2\frac{1}{4\epsilon^2}-C_0\frac{1}{2\epsilon}-C_0)+C_0e.
\end{eqnarray*}
Therefore, for $\epsilon_0>0$
small, we can make $\frac{L_1 \psi}{\psi}<-1$ and $\frac{L_2 \psi}{\psi}<-1$.

We now prove (ii). Actually, from above we can see that $\frac{L_i\psi}{\psi}\to -\infty$ as $\epsilon\to 0$ $(i=1, 2)$. Moreover, $\frac{\partial g}{\partial v}(u, \eta_{1}, \nabla u)$ and $\frac{\partial
  f}{\partial u}(\xi_{2}, v_{\lambda}, \nabla v)$ are both bounded on $\Omega$ for all $-a<\lambda<0$. Therefore, (ii) holds when $\epsilon_0$ is sufficiently small.
\end{proof}

\section{Proof of Theorem \ref{thm:main}}\label{sect3}

We now use the moving planes method to prove Theorem \ref{thm:main}. It consists of three steps:\\
\textbf{Step 1.} Prove $U(x, \lambda)\leq 0$ and $V(x, \lambda) \leq 0$ in $\Sigma_\lambda$ for $0<\lambda+a$ small.\\
\textbf{Step 2.} Prove $U(x, \lambda)\leq 0$ and $V(x, \lambda) \leq 0$ in $\Sigma_\lambda$ for all $\lambda$ in $(-a, 0)$.\\
\textbf{Step 3.} Prove $u_1(x)<0$ and $v_1(x)<0$, for all $(x_1, x_2)\in \Omega , x_1<0$.

\textbf{Proof of Step 1.} Let $\epsilon_0>0$ be a constant which we will choose later.
Assume for contraction that there exists a $\lambda_1\in (-a, -a+\epsilon_0]$ and a point $y_1\in \Sigma(\lambda_1)$ such that $U(y_1, \lambda_1)>0$ i.e. $\bar{U}(y_1, \lambda_1)>0$. We can choose $(y_1, \lambda_1)$ such  that
$$
\bar{U}(y_1, \lambda_1)=\max_{(x, \lambda)\in \overline{Q_{-a+\epsilon_0}} } \bar{U}(x, \lambda)>0.
$$

We claim that at $(y_1, \lambda_1)$, (\ref{eq:bbu}) holds. To see this, we should show $\frac{\partial u_{\lambda_1}}{\partial x_1}(y_1)\leq 0$. From the definition of $\bar{U}$ and from the boundary condition (\ref{con:bd}), we have
\begin{equation}\label{res:bd}
\bar{U}(x_1, \lambda)=0\  on\  x_1=\lambda\quad and \quad
\bar{U}\leq 0\ on\ x\in \partial \Omega.
\end{equation}
So we get
$$
\frac{\partial \bar{U}}{\partial \lambda}(y_1, \lambda_1)\geq 0, \quad i.e.\quad \frac{\partial U}{\partial \lambda}(y_1, \lambda_1)\geq 0.
$$
Thus,
\begin{equation*}
0 \leq  \frac{\partial U}{\partial \lambda}(y_1, \lambda_1)
=  2\frac{\partial u}{\partial x_1}(y_1^{\lambda_1})
= -2\frac{\partial u_{\lambda_1}}{\partial x_1}(y_1).
\end{equation*}
Therefore, (\ref{eq:bbu}) holds at point $y_1$. Moreover, from (\ref{res:bd}) we can see $y_1$ in the interior of $\Sigma(\lambda_1)$. So there hold the followings:
$$
(D^2_{ij}\bar{U}(y_1,\lambda_1))\leq 0,\quad \nabla \bar{U}(y_1,\lambda_1)=0.
$$
Inequality (\ref{eq:bbu}) turns to be:
\begin{equation}\label{eq:aaa}
\frac{L_1 \psi}{\psi}(y_1)\bar{U}(y_1, \lambda_1) +\frac{\partial g}{\partial v}(u(y_1), \eta_{1}(y_1, \lambda_1), \nabla u(y_1))\bar{V}(y_1, \lambda_1)\geq 0
\end{equation}
where
\begin{eqnarray*}
L_1 & = & a_{ij}(x, \lambda_1)D_{ij}+\frac{\partial g}{\partial p_1}(u_{\lambda_1}, v_{\lambda_1},\theta_1(x,\lambda_1) ,u_{\lambda_1,2})D_1\\
&   &  +\frac{\partial g}{\partial p_2}(u_{\lambda_1}, v_{\lambda_1},u_1 ,\theta_2(x,\lambda_1))D_2
+\frac{\partial g}{\partial u}(\xi_1(x,\lambda_1), v_{\lambda_1}, \nabla u ).
\end{eqnarray*}
From Lemma \ref{lem:psi}, we can choose $\epsilon_0$ sufficiently small, such that $\frac{L_1 \psi}{\psi}<-1$.
Since $\bar{U}(y_1, \lambda_1)>0$, it follows from (\ref{eq:aaa}) that
$$
\frac{\partial g}{\partial v}(u(y_1), \eta_{1}(y_1, \lambda_1), \nabla u(y_1))\bar{V}(y_1, \lambda_1)> 0.
$$
Combining this with assumption (\ref{con:monfg}), we have $\bar{V}(y_1, \lambda_1)>0$. Therefore, we may take
$$
\bar{V}(y_2, \lambda_2)=\max_{(x, \lambda)\in \overline{Q_{-a+\epsilon_0}}} \bar{V}(x, \lambda)>0.
$$
Now using (\ref{eq:bbv}), we can repeat the above argument and show that $\bar{U}(y_2, \lambda_2)>0$ and
\begin{equation}\label{eq:bbb}
\frac{L_2 \psi}{\psi}(y_2)\bar{V}(y_2, \lambda_2)+
\frac{\partial
f}{\partial u}(\xi_{2}(y_2, \lambda_2), v_{\lambda_2}(y_2), \nabla v(y_2))\bar{U}(y_2, \lambda_2)\geq 0,
\end{equation}
where
\begin{eqnarray*}
L_2 & = & b_{ij}(x, \lambda_2)D_{ij}+\frac{\partial f}{\partial p_1}(u_{\lambda_2}, v_{\lambda_2}, \tau_1(x,\lambda_2) ,v_{\lambda_2,2})D_1\\
& \quad & +\frac{\partial f}{\partial p_2}(u_{\lambda_2}, v_{\lambda_2},
v_{1}, \tau_{2}(x,\lambda_2)D_2+\frac{\partial
f}{\partial v}(u, \eta_{2}(x, \lambda_2), \nabla v).
\end{eqnarray*}
Let us put
$$
\alpha(y_1, \lambda_1)=\frac{L_1 \psi}{\psi}(y_1)<0, \quad \beta(y_1, \lambda_1)=\frac{\partial g}{\partial v}(u(y_1), \eta_{1}(y_1, \lambda_1), \nabla u(y_1))>0,
$$
$$
\gamma(y_2, \lambda_2)=\frac{L_2 \psi}{\psi}(y_2)\bar{V}(y_2, \lambda_2)<0, \quad \delta(y_2, \lambda_2)=\frac{\partial
f}{\partial u}(\xi_{2}(y_2, \lambda_2), v_{\lambda_2}(y_2), \nabla v(y_2))>0.
$$
From (\ref{eq:aaa}) and (\ref{eq:bbb}), we have
\begin{eqnarray*}
\bar{U}(y_1, \lambda_1) & \leq & -\frac{\beta(y_1, \lambda_1)}{\alpha(y_1, \lambda_1)}\bar{V}(y_1, \lambda_1)\\
& \leq  & -\frac{\beta(y_1, \lambda_1)}{\alpha(y_1, \lambda_1)}\bar{V}(y_2, \lambda_2)\\
& \leq & \frac{\beta(y_1, \lambda_1)\delta(y_2, \lambda_2)}{\alpha(y_1, \lambda_1)\gamma(y_2, \lambda_2)}\bar{U}(y_2, \lambda_2)\\
& \leq & \frac{\beta(y_1, \lambda_1)\delta(y_2, \lambda_2)}{\alpha(y_1, \lambda_1)\gamma(y_2, \lambda_2)}\bar{U}(y_1, \lambda_1).
\end{eqnarray*}
Then we obtain
\begin{equation}\label{eq:ccc}
\alpha(y_1, \lambda_1)\gamma(y_2, \lambda_2)-\beta(y_1, \lambda_1)\delta(y_2, \lambda_2)\leq 0.
\end{equation}
But from Lemma \ref{lem:psi} (ii), we know that if we choose $\epsilon_0$ sufficiently small, (\ref{eq:ccc}) does not hold.
\textbf{(Step 1. is completed.)}
\\
Let us put
$$
\bar{\mu}=\sup\{\ -a<\mu\leq 0\ |\ U(x,\lambda)\leq 0, V(x, \lambda)\leq 0, \forall x\in \Sigma(\lambda), \forall \lambda\leq \mu\}.
$$
We claim that $\bar{\mu}=0$.

\textbf{Proof of Step 2.}
Suppose the contrary, that $\bar{\mu}<0$. The maximality of $\bar{\mu}$ tells us that there exist sequences $\{\lambda^k\}$ and $x^k\in \Sigma(\lambda^k)$ satisfy:
$$
\bar{\mu}<\lambda^k<0,\quad  \lim_{k\to \infty}\lambda^k=\bar{\mu};\quad
U(x^k, \lambda^k)>0 \ or\ V(x^k, \lambda^k)>0.
$$
Without loss of generality, we may assume
$$
U(x^k, \lambda^k)=\max_{x\in \overline{\Sigma(\lambda^k)}}U(x, \lambda^k)>0; \quad \lim_{k\to \infty}x^k=\bar{x}\in \overline{\Sigma(\bar{\mu})}.
$$
On one hand the boundary condition (\ref{con:bd}) tells us that $U(x, \lambda^k)\leq 0$, $x\in \partial \Sigma(\lambda^k)\setminus T_{\lambda^k}$. On the other hand when $x\in T_{\lambda^k}$, $U(x, \lambda^k)=0$. Therefore, the point $x^k$ must be an interior point of $\Sigma(\lambda^k)$, which implies:
\begin{equation*}
(D^2U(x^k, \lambda^k))\leq 0,\quad \nabla U(x^k, \lambda^k)=0,\quad U(x^k, \lambda^k)>0.
\end{equation*}
Taking limit $k\to \infty$ and using the fact $
U(\bar{x}, \bar{\mu})\leq 0
$, we obtain

\begin{equation}
\label{eq:ddd}
(D^2U(\bar{x}, \bar{\mu}))\leq 0,\quad \nabla U(\bar{x}, \bar{\mu})=0,\quad U(\bar{x}, \bar{\mu})= 0.
\end{equation}

We now derive the differential equation $U(x, \bar{\mu})$ satisfies.
From the definition of $\bar{\mu}$, we can check that
$$
u_1(x)\leq 0, \quad \forall x\in \Sigma(\bar{\mu}).
$$
Using (\ref{con:symfg}), we have
$$
g(u(x), v(x), u_1(x), u_2(x))\geq g(u(x), v(x), -u_1(x), u_2(x))\quad \forall x\in \Sigma(\bar{\mu}).
$$
Thus,
\begin{eqnarray*}
a_{ij}(x)U_{ij} (x, \bar{\mu})& = & det(D^2 u_{\bar{\mu}})(x)-det(D^2 u)(x)\\
& = & det(D^2 u)(x^{\bar{\mu}})-det(D^2 u)(x)\\
& = & -g(u(x^{\bar{\mu}}), v(x^{\bar{\mu}}), \nabla u(x^{\bar{\mu}}))+g(u(x), v(x), \nabla u (x))\\
& \geq & -g(u(x^{\bar{\mu}}), v(x^{\bar{\mu}}), \nabla u(x^{\bar{\mu}}))+g(u(x), v(x), -u_1 (x), u_2(x))\\
& = & -g(u_{\bar{\mu}}(x), v_{\bar{\mu}}(x), -u_{\bar{\mu},1}(x), u_{\bar{\mu},2}(x))\\
&   & +g(u(x), v(x), -u_1 (x), u_2(x)).
\end{eqnarray*}
where $(a_{ij})=\frac{1}{2}(det(D^2 u_{\bar{\mu}})(D^2
u_{\bar{\mu}})^{-1}+det(D^2 u)(D^2 u)^{-1})$.
Using the Taylor's expansion, $U(x, \bar{\mu})$ satisfies: for all $x\in \Sigma(\bar{\mu})$
\begin{eqnarray*}
  \nonumber
  a_{ij}(x)U_{ij}(x,\bar{\mu})&-&\frac{\partial g}{\partial p_1}(u_{\bar{\mu}}, v_{\bar{\mu}}, -\theta_1(x,\bar{\mu}) ,
  u_{\bar{\mu},2})U_{1}
  +\frac{\partial g}{\partial p_2}(u_{\bar{\mu}}, v_{\bar{\mu}},
  -u_{1}, \theta_{2}(x,\bar{\mu}))U_{2}\\
\nonumber  &+&\frac{\partial g}{\partial u}(\xi(x,\bar{\mu}), v_{\bar{\mu}}, -u_1, u_2 )U+\frac{\partial
  g}{\partial v}(u, \eta_{1}(x, \bar{\mu}), -u_1, u_2)V(x,\bar{\mu})\geq 0;
\end{eqnarray*}
Since $V(x,\bar{\mu})\leq 0$, $\forall x\in \Sigma(\bar{\mu})$ and $\frac{\partial
  g}{\partial v}>0$, we can rewrite the above inequality as
\begin{eqnarray*}
  \nonumber
  a_{ij}(x)U_{ij}(x,\bar{\mu})&-&\frac{\partial g}{\partial p_1}(u_{\bar{\mu}}, v_{\bar{\mu}}, -\theta_1(x,\bar{\mu}) ,
  u_{\bar{\mu},2})U_{1}
  +\frac{\partial g}{\partial p_2}(u_{\bar{\mu}}, v_{\bar{\mu}},
  -u_{1}, \theta_{2}(x,\bar{\mu}))U_{2}\\
\nonumber  &+&\frac{\partial g}{\partial u}(\xi(x,\bar{\mu}), v_{\bar{\mu}}, -u_1, u_2 )U\geq 0.
\end{eqnarray*}
Hopf's lemma implies $\bar{x}\in \partial \Sigma(\bar{\mu})$ and
\begin{equation}
\label{eq:eee}\frac{\partial U}{\partial \nu}<0,
\end{equation}
where $\nu$ is the outward normal on $\partial \Sigma(\bar{\mu})$.
(\ref{eq:eee}) is impossible since we already have (\ref{eq:ddd}).
\textbf{Step 2. is completed.}

\textbf{Proof of Step 3.}
Form Step 2., we have obtained that $U(x, \lambda)\leq 0$ and $V(x, \lambda) \leq 0$ in $\Sigma_\lambda$ for all $\lambda \in (-a, 0)$. Similar to the computation in Step 2., we can obtain: $\forall \lambda \in (-a, 0)$
\begin{eqnarray}
  \nonumber
  a_{ij}(x)U_{ij}(x,\lambda)&-&\frac{\partial g}{\partial p_1}(u_\lambda, v_\lambda, -\theta_1(x,\lambda) ,
  u_{\lambda,2})U_{1}
  +\frac{\partial g}{\partial p_2}(u_\lambda, v_{\lambda},
  -u_{1}, \theta_{2}(x,\lambda))U_{2}\\
\nonumber  &+&\frac{\partial g}{\partial u}(\xi(x,\lambda), v_{\lambda}, -u_1, u_2 )U\geq 0,
\end{eqnarray}
where $
\xi(x, \lambda)\in (\overline{u(x),
u_{\lambda}(x)}$ and
$\theta_{i}(x, \lambda)\in (\overline {u_{i}(x),
u_{\lambda,i}(x)})
$.
Form using the fact that $U_\lambda\leq 0$ on $\partial \Sigma(\lambda)$ and $U_\lambda=0$ on $T_\lambda$ and by using the Hopf's lemma, we have $\frac{\partial U}{\partial x_1}>0$ on $T_\lambda$, i.e. $-2u_1(x_1, x_2)>0$ for $x_1=\lambda$.
Therefore, $u_1(x_1, x_2)<0$ for all $x_1<0$.
Similarly, for $V_\lambda$:
\begin{eqnarray*}
  \nonumber
  b_{ij}(x)V_{ij}(x,\lambda)&-&\frac{\partial f}{\partial p_1}(u_\lambda, v_\lambda, -\tau_1(x,\lambda) ,
  v_{\lambda,2})V_{1}
  +\frac{\partial f}{\partial p_2}(u_\lambda, v_{\lambda},
  -v_{1}, \tau_{2}(x,\lambda))V_{2}\\
\nonumber  &+&\frac{\partial f}{\partial v}(u, \eta, -v_1, v_2 )V\geq 0,
\end{eqnarray*}
where $\tau_i(x, \lambda)\in (\overline{v_{i}(x), v_{\lambda,i}(x)})$, $\eta(x, \lambda) \in  (\overline {v(x), v_{\lambda}(x)})$.
By using the above inequality, we can also get $v_1(x_1, x_2)<0$ for all $x_1<0$.
\textbf{(Step 3. is completed.)}

\section{Proof of Corollary \ref{thm:col}}\label{sect4}
Corollary \ref{thm:col} is a simple consequence of Theorem \ref{thm:main}. Applying Theorem \ref{thm:main} directly to the solution $(u, v)$, we obtain
\begin{equation}\label{eq:fff}
u(x_1, x_2)\geq u(-x_1, x_2),\  v(x_1, x_2)\geq v(-x_1, x_2), \ x_1<0, \ (x_1, x_2)\in \Omega.
\end{equation}
We consider functions:
$$
\tilde{u}(x_1, x_2)=u(-x_1, x_2)\quad \tilde{v}(x_1, x_2)=v(-x_1, x_2), \quad (x_1, x_2)\in \Omega.
$$
Since functions $f$ and $g$ are symmetric in $p_1$ and $\Omega$ is symmetric in $x_1$, $(\tilde{u}, \tilde{v})$ is also a solution of (\ref{eq:sys}). Thus we can apply Theorem \ref{thm:main} to $(\tilde{u}, \tilde{v})$ which tells us that
$$
\tilde{u}(x_1, x_2)\leq \tilde{u}(-x_1, x_2),\quad \tilde{v}(x_1, x_2)\leq \tilde{v}(-x_1, x_2),\quad  x_1<0, (x_1, x_2)\in \Omega,
$$
or equivalently
\begin{equation}\label{eq:ggg}
u(-x_1, x_2)\leq u(x_1, x_2),\  v(-x_1, x_2)\leq v(x_1, x_2),\  x_1<0,\ (x_1, x_2)\in \Omega.
\end{equation}
(\ref{eq:fff}) and (\ref{eq:ggg}) together give us
$$
u(-x_1, x_2)=u(x_1, x_2),\quad v(-x_1, x_2)=v(x_1, x_2),\quad  \forall (x_1, x_2)\in \Omega.
$$
Therefore $u$ and $v$ are both symmetric in $x_1$ and Corollary \ref{thm:col} is proved.

\end{document}